\documentclass[12pt]{article}
\usepackage{sigsam, amsmath,amsthm,amssymb}

\newtheorem{thm}{Theorem}[section]

\theoremstyle{definition}
\newtheorem{example}[thm]{Example}
\newtheorem{lem}[thm]{Lemma}

\theoremstyle{definition}

\newtheorem{rem}[thm]{Remark}
\numberwithin{equation}{section}
 \usepackage{tikz}
\issue{TBA}
\articlehead{TBA}
\titlehead{Exploiting Symmetry in the Power Flow Equations Using Monodromy}
\authorhead{Julia Lindberg, Nigel Boston, Bernard C. Lesieutre}
\begin{document}

\title{Exploiting Symmetry in the Power Flow Equations Using Monodromy}

\author{Julia Lindberg${^1}$, Nigel Boston${^{1,2}}$, Bernard C. Lesieutre${^1}$ \\
Department of Electrical and Computer Engineering${^1}$, Department of Mathematics${^2}$ \\
University of Wisconsin-Madison \\
Madison, USA, 53706 \\
\url{jrlindberg@wisc.edu,nboston@wisc.edu,lesieutre@wisc.edu}}

\date{}

\maketitle

\begin{abstract}
We propose solving the power flow equations using monodromy. We prove the variety under consideration decomposes into trivial and nontrivial subvarieties and that the nontrivial subvariety is irreducible. We also show various symmetries in the solutions. We finish by giving numerical results comparing monodromy against polyhedral and total degree homotopy methods and giving an example of a network where we can find all solutions to the power flow equation using monodromy where other homotopy techniques fail. This work gives hope that finding all solutions to the power flow equations for networks of realistic size is possible.
\end{abstract}

\section{Introduction}
The \textit{optimal power flow problem} is a nonconvex quadratically constrained quadratic program that seeks to minimize generation costs of electricity subject to demand constraints, as well as physical and engineering constraints. The \textit{power flow equations} are a system of quadratic equations that give the nonconvex constraints in the optimal power flow problem. Recent work has used techniques from algebraic geometry to find all solutions to the power flow equations for a fixed network \cite{mehta2016numerical, zachariah2017efficiently, chen2018on}, as well as study the distribution of the number of real solutions \cite{lindberg2018the, lesieutre2019on,zachariah2018distributions}. While there has been some success using numerical and symbolic techniques, these methods don't take advantage of any of the symmetry in the power flow equations. In this paper we apply monodromy techniques to find all complex solutions to the power flow equations for networks of varying size. 
% In Section \ref{2} we present the system of equations under consideration, in Section \ref{3} we prove the irreducibility of one subvariety of the variety defined by the power flow equations as well as proofs of different symmetry present in the equations.
% We finish in Section \ref{5} with numerical results comparing monodromy techniques with polyhedral and total degree homotopy techniques.

\section{The Power Flow Equations}\label{2}
We model an $n$-node electric power network as a biconnected\footnote{A graph is \textit{biconnected} if removing any vertex does not disconnect the graph. For a graph that is not biconnected, the power flow equations decouple so they are analyzed separately on smaller graphs.}, undirected graph, $G = (V, E)$, where each vertex represents a node in the power network. There is an edge, $e_{km}$ between vertices $v_k$ and $v_m$ if the corresponding nodes in the power network are connected. Each edge has a known \textit{susceptance} $b_{km} \in \mathbb{R}$. At each node, $k$, the relationship between the active power flows is captured by the nonlinear relations
\begin{equation}
\begin{aligned}
    x_k^2 + y_k^2 &= 1 \\
    \sum_{m=0}^{n-1} b_{km}(x_ky_m - x_m y_k) &= P_k \label{pfeq2}
\end{aligned}
\end{equation}
where $P_k \in \mathbb{R}$ are the \textit{active power injections} for $k = 1,\ldots,n-1$. We fix $v_0$ to be the \textit{reference node}, meaning $x_0 = 1$ and $y_0 = 0$. Under these assumptions our power system is \textit{lossless} and it has all \textit{PV nodes}. We assume our network has \textit{zero power injections} so $P_k = 0$ for all $k = 1, \ldots, n-1$. Under these assumptions the parameters are the susceptances $b_{km}$ and the variables are $x_k,y_k$ for $k=1,\ldots,n-1$.  For fixed $b_{km} \in \mathbb{R}$ we aim to find all complex solutions to \eqref{pfeq2}.

\section{Symmetry in Solutions} \label{3}
We briefly outline the idea of monodromy below using the same notation as \cite{duff2019solving}, but give \cite{duff2019solving, amendola2016solving, campo2017critical}, as more complete references. Let $F_b$ be a parameterized polynomial system in $N$ variables and call the space of all such polynomial systems $B$. Assume the solution set of $F_b$ is zero dimensional. Let $\mathcal{V}$ denote the \textit{solution variety} of $F_b$ i.e. $\mathcal{V} = \{(F_b, x) \in B \times \mathbb{C}^N : F_b(x) = 0\}$. Consider the projection $\pi : \mathcal{V} \to B$ that maps a pair $(F_b,x) \mapsto F_b$ and the fibers $\pi^{-1}(F_b) = \{x \in \mathbb{C}^N : F_b(x) = 0 \}$. For almost all choices of parameters in $B$, $|\pi^{-1}(F_b)| = K$ is constant. Define $D$ to be the \textit{discriminant locus} of $F_b$, this is the set of measure zero in $B$ where $|\pi^{-1}(F_b)| \neq K$. We define the \textit{fundamental group} $\pi_1(B \backslash D)$ as a set of loops modulo homotopy equivalence that start and finish at a point $b \in B \backslash D$. Each loop permutes elements in $\pi^{-1}(F_b)$ and induces a group action called the \textit{monodromy action}. Monodromy methods work by taking one solution $\hat{x}$ to the system of equations $F_{\hat{b}}$ and finding other elements of $\pi^{-1}(F_{\hat{b}})$ via the monodromy action. The monodromy action is transitive if and only if the variety $V(F_b)$ is irreducible.

Under the assumptions given in Section \ref{2}, for any graph $G$, equations \eqref{pfeq2} have $2^{n-1}$ \textit{trivial solutions} of the form $(x_k,y_k) = (\pm 1, 0)$. Therefore, our problem is reduced to finding all nontrivial solutions of \eqref{pfeq2}. In order to apply monodromy and have any hope of finding all nontrivial solutions we need the nontrivial solutions to form an irreducible variety.
\begin{lem} \label{irreduciblelemma}
 The nontrivial solutions of \eqref{pfeq2} form an irreducible variety.
\end{lem}{}
\begin{proof}
By Theorem $6$ of \cite{chen2018counting}, the nontrivial component of \eqref{pfeq2} for tree networks is empty, so this statement is vacuously true. Consider the change of variables $x_i = \frac{2t_i}{1+t_i^2}$ and $y_i = \frac{1-t_i^2}{1+t_i^2}$. This gives a new system of equations for $k=1,\ldots,n-1$
\begin{align}
    0 &= \sum_{m=0}^{n-1} b_{km} \Big( \frac{2t_k(1-t_m^2)-2t_m(1-t_k^2)}{(1+t_k^2)(1+t_m^2)} \Big). \label{paramequs}
\end{align}{}

\noindent By Remark $2$ of \cite{duff2019solving} is suffices to show that the following map has dense image:
\begin{align*}
    \pi &: \mathcal{V} \to \mathbb{C}^{n-1} \\
    &(F_b, t) \mapsto t
\end{align*}{}
where $F_b$ is the system of equations defined in \eqref{paramequs} and $t = (t_1,\ldots, t_{n-1})$. For all $t_k \in \mathbb{C} \backslash \{ \pm \sqrt{-1},0, \pm 1 \} $, $k = 1,\ldots, n-1$, this gives a linear system of $n-1$ equations in $|E|$ unknowns where the unknowns are the susceptances $b_{km}$. Since we do not consider trees, $|E| \geq n$. Let $b \in \mathbb{R}^{|E|}$ be the vector of susceptances. Then this linear system can be written as $Ab = 0$ where $A\in \mathbb{C}^{n-1 \times |E|}$ is  a weighted incidence matrix of $G$ with the first row removed. This matrix  has rank $n-1$ so long as none of the weights are zero, which occurs for all $t_k, t_m \not\in \{\pm \sqrt{-1}, 0 , \pm 1 \}, t_k \neq t_m$. Therefore, for generic $t \in \mathbb{C}^{n-1}$ we can find a nonzero solution $b$ to $(\ref{paramequs})$ giving that the map $(F_b, t) \mapsto t$ is dense in $\mathbb{C}^{n-1}$.
\end{proof}{}

In addition to ignoring the trivial component, we also wish to exploit the symmetry of \eqref{pfeq2}.

\begin{lem}
If $(x_1,\ldots,x_{n-1},y_1,\ldots,y_{n-1})$ is a solution to \eqref{pfeq2}, so is $(x_1,\ldots,x_{n-1},-y_1,\ldots,-y_{n-1})$
\end{lem}
\begin{proof}
Substituting in $(x_1,\ldots,x_{n-1},-y_1,\ldots,-y_{n-1})$ to \eqref{pfeq2} the result is immediate.
\end{proof}

\begin{lem}
Let $G = (V,E)$ be a bipartite graph with disjoint vertex sets $S,T \subset V$ that partition $V$ where for all $e = v_mv_n \in E$, $v_m \in S$ and $v_n \in T$. Without loss of generality, say $v_1,\ldots,v_s \in S$ and $v_{s+1},\ldots,v_{n-1} \in T$. If $(x_1,\ldots,x_{n-1},y_1,\ldots,y_{n-1})$ is a solution to \eqref{pfeq2} so is
\begin{enumerate}
    \item $(x_1,\ldots,x_{n-1},-y_1,\ldots,-y_{n-1})$
    \item $(-x_1,\ldots,-x_s,x_{s+1},\ldots,x_{n-1},y_1,\ldots,y_s,-y_{s+1},\ldots,-y_{n-1})$
    \item $(-x_1,\ldots,-x_s, x_{s+1},\ldots,x_{n-1},-y_1,\ldots,-y_s,y_{s+1},\ldots,y_{n-1})$
\end{enumerate}
\end{lem}
\begin{proof}
At a node $k \in S$ the power flow equations are
\begin{align}
    0 &= \sum_{m=s+1}^{n-1} b_{km}(x_m y_k - x_k y_m).
\end{align}
At a node $l \in T$ the power flow equations are
\begin{align}
   0 &= \sum_{m=1}^sb_{lm}(x_my_l-x_ly_m).
\end{align}
Substituting in $(1)-(3)$ to the two expressions above, the result is clear.
\end{proof}

\section{Numerical Simulations}\label{5}
For all numerical computations we use the package \texttt{HomotopyContinuation.jl} \cite{HomotopyContinuationjulia}. Tables~\ref{completepaths} and \ref{cyclicpaths} show the average number of loops tracked to find all solutions using monodromy in comparison with the number of paths needed to track for polyhedral and total degree homotopy continuation methods. These tables also show the time it took to find all solutions. In each case, monodromy is the fastest method to find all solutions.

\begin{table}[h!]
\centering
\begin{tabular}{ |c|c|c|c|c|c|c|c| }
\hline 
 & $K_4$ & $K_5$ & $K_6$ & $K_7$ & $K_8$ & $K_9$ & $K_{10}$ \\
 \hline
 $\#$ of loops: monodromy & $24.4$ & $94.8$ & $470.4$ & $1669.2$ & $7915$ & $25112.2$ & $95829.2$ \\ 
 $\#$ of paths: polyhedral & $40$ & $192$ & $864$ & $3712$ & $15488$ & $63488$ & $257536$ \\ 
$\#$ of paths: total degree & $64$ & $256$ & $1024$ & $4096$ & $16384$ & $65536$ & $262144$ \\ 
 \hline
 time (s): monodromy & $0.01$ & $0.05$ &$0.37$  & $1.97$ & $16.39$ & $65.33$ & $357.926$ \\
 time (s): polyhedral & $0.06$ & $0.37$ & $2.53$ & $17.10$ & $112.43$ & $609.49$ & $2637.22$ \\
 time (s): total degree & $0.04$ & $0.21$ & $1.45$ & $8.17$ & $48.78$ & $329.60$ & $1510.01$  \\
 \hline
\end{tabular}
\caption{Numerical results to find all solutions for complete networks}
\label{completepaths}
\end{table}

\begin{table}[h!]
\centering
\begin{tabular}{ |c|c|c|c|c|c|c| }
\hline 
 & $C_5$ & $C_6$ & $C_7$ & $C_8$ & $C_9$ & $C_{10}$ \\
 \hline
 $\#$ of loops: monodromy & $181.6$ & $176.6$ & $962.0$ & $921.6$ & $1110.6$ & $752$  \\ 
 $\#$ of paths: polyhedral & $80$& $256$ & $832$ & $2688$ & $8704$ & $28160$ \\ 
$\#$ of paths: total degree & $256$ & $1024$ & $4096$ & $16384$ & $65536$ & $262144$  \\ 
 \hline
 time (s): monodromy & $0.13$ & $0.158$ &$1.10$ & $1.46$ & $2.48$ & $2.60$   \\
 time (s): polyhedral & $2.7$ & $3.03$ & $5.37$ & $14.8$ & $56.36$ & $211.24$  \\
 time (s): total degree &  $2.11$& $3.40$ & $9.76$ & $31.91$ & $200.41$ & $862.50$  \\
 \hline
\end{tabular}
\caption{Numerical results to find all solutions for cyclic networks}
\label{cyclicpaths}
\end{table}

\begin{rem}
Another advantage to monodromy is that we don't lose solutions as we track them from a start to target system. We observed in small networks that all methods found all solutions but for larger networks polyhedral and total degree homotopy methods lost solutions.
\end{rem}
 
\begin{example}
A final benefit of monodromy when applied to the power flow equations is that it can find all complex solutions when other methods can't. We consider the cyclic graph on $20$ vertices. This system has $1,847,560$ complex solutions but ignoring the trivial solutions and up to symmetry it has $330,818$. If we tried to use total degree homotopy on this system, the Bezout bound is $274,877,906,944$ so we would have to track over $274$ billion paths. In addition, polyhedral methods aren't practical as the solver could not find a start system. Using monodromy we found all $330,818$ solutions in $15,375$ seconds after tracking 792,934 loops. This example is the largest network to the authors' knowledge for which all solutions to the power flow equations have been found for a power system model.\footnote{The authors' note that in \cite{coss2018locating} all real solutions to a network on $60$ vertices were found, but as noted by the authors in \cite{coss2018locating}, the assumptions in that paper are not attainable by any realistic power systems model.}
\end{example}

\section{Conclusion}
 
In this note we applied monodromy methods to the power flow equations with great results. Monodromy methods gain tremendous computational speed-up by decomposing the variety into trivial and nontrivial components and solving up to symmetry. Finally, we are able to push the current computational limits and find all solutions to the power flow equations for the cylic graph on $20$ nodes, the largest power network in which all solutions have been found to date.

\section*{Acknowledgements}
The authors' gratefully thank Jose Israel Rodriguez for his helpful comments and insight.

\bibliographystyle{unsrt}

\bibliography{main}

\end{document}